\documentclass[11pt,twoside]{preprint}
\usepackage[a4paper,innermargin=1.4in,outermargin=1.4in,bottom=1.5in,marginparwidth=1.5in,marginparsep=3mm]{geometry}
\usepackage{amsmath,amsthm,amssymb,enumerate}
\usepackage{hyperref}
\usepackage{mathrsfs}
\usepackage{breakurl}

\usepackage[osf]{newtxtext}
\usepackage[varqu,varl]{inconsolata}
\usepackage[cal=boondoxo]{mathalfa}
\usepackage[bigdelims,vvarbb,cmintegrals]{newtxmath}

\usepackage{xcolor}
\usepackage{mhequ}
\usepackage{microtype}

\newtheorem{theorem}{Theorem}[section]
\newtheorem{lemma}[theorem]{Lemma}
\newtheorem{proposition}[theorem]{Proposition}

\theoremstyle{definition}

\theoremstyle{remark}
\newtheorem{remark}[theorem]{Remark}

\newcommand{\vn}[1]{{\vert\kern-0.23ex\vert\kern-0.23ex\vert #1 
    \vert\kern-0.23ex\vert\kern-0.23ex\vert}}

\newcommand{\stoch}[1]{[ #1\vert\kern-0.5ex]}

\allowdisplaybreaks

\newcommand\Z{\mathbb{Z}}

\newcommand\cA{\mathcal{A}}

\newcommand\cS{\mathcal{S}}
\newcommand\cG{\mathcal{G}}

\newcommand{\cT}{\mathcal{T}}

\newcommand\cP{\mathcal{P}}
\newcommand\cF{\mathcal{F}}

\newcommand{\F}{\mathcal{F}}
\newcommand{\A}{\mathcal{A}}
\newcommand{\G}{\mathcal{G}}
\newcommand{\bA}{\mathbb{A}}
\newcommand{\bB}{\mathbb{B}}

\newcommand\scC{\mathscr{C}}

\def\eps{\varepsilon}

\newcommand{\R}{{\mathbb{R}}}

\newcommand{\E}{\mathbf{E}}
\newcommand{\N}{\mathbb{N}}

\newcommand{\floorH}{\lfloor H \rfloor}

\def\d{\partial}

\title{Regularisation by regular noise}
\author{M\'at\'e Gerencs\'er
}
\institute{TU Wien}
\begin{document}
\maketitle
\begin{abstract}
We show that perturbing ill-posed differential equations with (potentially very) smooth random processes can restore well-posedness -- even if the perturbation is (potentially much) more regular than the drift component of the solution.
The noise considered is of fractional Brownian type, and the familiar regularity condition $\alpha>1-1/(2H)$ is recovered for all non-integer $H>1$.
\end{abstract}

\section{Introduction}
Consider the stochastic differential equation
\begin{equ}\label{eq:main}
X_t=\int_0^tb(X_r)\,dr+B^H_t.
\end{equ}
For Hurst parameter $H\in(0,1)$, fractional Brownian motions $B^H$ can be defined via the Mandelbrot - van Ness representation \cite{MvN}
\begin{equs}\label{eq:Mandelbrot}
B^H_t = \int_{-\infty}^0 \bigl(|t-s|^{H-1/2}- |s|^{H-1/2}\bigr) \, dW_s + \int_0^t |t-s|^{H-1/2} \, dW_s,
\end{equs}
where $W$ is a two-sided $d$-dimensional standard Brownian motion on some probability space $(\Omega,\cF,\mathbb{P})$.
The complete filtration generated by the increments of $W$ is denoted by $\mathbb{F}=(\F_t)_{t\in\R}$, with respect to which the notion of adaptedness is understood in the sequel, unless otherwise specified.
Since \eqref{eq:Mandelbrot} is a fractional integration, it is natural to extend this scale of random processes to parameters $H\in(1,\infty)\setminus\Z$ by
\begin{equ}\label{eq:BH def}
B^H_t=\int_{0\leq r_1\leq\cdots \leq r_{\floorH}\leq t}B^{H-\floorH}_{r_1}\,dr_1\cdots\,dr_{\floorH}.
\end{equ}
It is a well-studied phenomenon that adding a noise term to an ill-posed differential equation can provide regularisation effects. That is, while for $\alpha<1$ and $b\in C^\alpha$
the equation $X'_t=b(X_t)$ is in general not well-posed, upon perturbing it with some random process (Brownian, fractional Brownian, or L\'evy-tpye), due to the oscillations of the noise one obtains a well-posed equation.
It might seem counterintuitive to expect similar effects from the processes $B^H$ for large $H$. Indeed, in the regime\footnote{Under the condition \eqref{eq:true exponent} of the main theorem, this phenomenon may already happen for $H>1+1/\sqrt{2}$.} $H>2$ this would mean that the noise component of \eqref{eq:main} would be more regular (or, less oscillatory) than the drift component.
Therefore it may be surprising that
nontrivial regularisation effects happen for all $H\in(1,\infty)\setminus\Z$.

Before stating the main result, let us clarify the solution concept: given a stopping time $\tau$, a \emph{strong solution} of \eqref{eq:main} up to time $\tau$ is an adapted process $X$ such that the equality \eqref{eq:main} holds almost surely for all $t\in [0,\tau]$. Two solutions are understood to be equivalent if they are equal almost surely for all $t\in [0,\tau]$.
We then have the following.
\begin{theorem}\label{thm:main}
Let $H\in[1/2,\infty)\setminus\Z$ and $b\in C^\alpha$, where
\begin{equ}\label{eq:true exponent}
\alpha>1-\frac{1}{2H}.
\end{equ}
Then \eqref{eq:main} has a unique strong solution up to time $1$.
\end{theorem}

The well-posedness of \eqref{eq:main} for $H\in(0,1)$ under the condition \eqref{eq:true exponent} is well-known in the literature.
This condition first appeared in \cite{NO1} in the scalar, $H\in(1/2,1)$ case and later in \cite{Cat-Gub} the the multidimensional, $H\in(0,1)$ case\footnote{Unlike the classical notion of uniqueness considered here, \cite{Cat-Gub} proves uniqueness in the so-called path-by-path sense, which is a stronger notion and does not easily follow from stochastic sewing.
}.
In the multidimensional, $H\in(1/2,1)$ case a simpler proof
was given in \cite{BDG},
relying on various sewing lemmas.
The present approach is simpler still (for example Girsanov transformation is fully avoided\footnote{This is not mere convenience but a necessity. One simply cannot expect the Girsanov transformation to work: the law of the solution is singular with respect to the law of $B^H$ as soon as $H>1+\alpha$.}), and uses only the stochastic sewing lemma introduced in \cite{Khoa}.
In the regime $H\in(0,1/2)$, the condition \eqref{eq:true exponent} allows for negative $\alpha$, that is, distributional drift, which makes the interpretation of the integral in \eqref{eq:main} nontrivial.
While this is also possible with stochastic sewing
(see e.g. \cite[Cor~4.5]{BDG} for a step towards this),
since the novelty of Theorem \ref{thm:main} is in the other direction ($H>1$), we do not pursue this here.

Instead of thinking of \eqref{eq:main} as an equation being driven by ``regular'' noise, one can equivalently view it as a ``degenerate'' SDE
\begin{equs}[eq:system]
dU^1_t&=dB^{H-\lfloor H\rfloor}_t\,;\\
dU^2_t&=U^1\,dt\,;\\
\,&\,\vdots\\
dX_t&=\big(b(X_t)+U^{\lfloor H\rfloor}_t\big)\,dt\,.
\end{equs}
For $H=k+1/2$, $k\in\N$, the driving noise in \eqref{eq:system} is standard Brownian, and therefore techniques based on the Zvonkin-Veretennikov transformation \cite{Zvonkin, Veretennikov} apply. This approach was explored recently in \cite{Brownian-1,Brownian-2}, where the authors prove well-posedness under the condition \eqref{eq:true exponent} with an elaborate PDE-based proof.

Let us conclude the introduction by spelling out some trivial equivalences. First, \eqref{eq:true exponent} can be rewritten as
\begin{equ}\label{eq:true exponent2}
1+H\alpha-H>1/2.
\end{equ}
Second, a solution of \eqref{eq:main} can be written as $X=\varphi+B^H$, where $\varphi$ is a fixed point of the map
\begin{equ}\label{eq:main-phi}
\varphi\mapsto \cT(\varphi),\qquad \big(\cT(\varphi)\big)_t=\int_0^t b(\varphi_r+B^H_r)\,dr.
\end{equ}
In Section \ref{sec:proof}, Theorem \ref{thm:main} is proved by showing that (a stopped version of) $\cT$ is a contraction on some space.
The main tools for the proof are set up in Section \ref{sec:pre}, more precisely in Lemmas \ref{lem:shifted SSL} and \ref{lem:regularity}.

\begin{remark}
A heuristic power counting argument for \eqref{eq:true exponent} can be given as follows.
For strong uniqueness one would like the map $\varphi\mapsto \cT(\varphi)$ to be Lipschitz with a small constant.
A more modest goal would be try to show that $t,x\mapsto\big(\cT(x)\big)_t$ is Lipschitz in $x\in\R^d$.
For starters, it is trivial that $t,x\mapsto\big(\cT(x)\big)_t$ is Lipschitz in $t$, $C^\alpha$ in $x$.
Thanks to the presence of $B^H$ one can use stochastic sewing to trade time and space regularity at the exchange rate $t^H\sim x$.
However, we can only trade as long as we remain with more than $1/2$ regularity in $t$ (this is the ``$1/2$ condition of the stochastic sewing lemma''), which means that the available gain in $x$ regularity is $1/(2H)$.
\end{remark}
\section{Preparations}\label{sec:pre}

\subsection{Notation and basics}\label{sec:notations}
The conditional expectation given $\F_s$ is denoted by $\E^s$. We often use the conditional Jensen's inequality (CJI) $\|\E^s X\|_{L_p(\Omega)}\leq\|X\|_{L_p(\Omega)}$ for $p\in[1,\infty]$ and the following elementary inequality for any $p\in[1,\infty]$, $X\in L_p(\Omega)$, and $\F_s$-measurable $Y$:
\begin{equ}\label{eq:conditional}
\|X-\E^s X\|_{L_p(\Omega)}\leq 2\|X-Y\|_{L_p(\Omega)}.
\end{equ}

For $0\leq s<t\leq 1$ and $\R^d$-valued functions $f$ on $[s,t]$ we introduce the (semi-)norms
\begin{equs}
\|f\|_{C^0[s,t]}&=\sup_{u\in[s,t]}|f_u|\,;& &\\
\,[f]_{C^\gamma[s,t]}&=\sup_{\substack{u,v\in[s,t]\\ u<v}}\frac{|\d^{\hat \gamma}f_u-\d^{\hat\gamma}f_v|}{|u-v|^{\bar\gamma}}\,,&\qquad&\gamma>0,\,\gamma=\hat\gamma+\bar\gamma,\,\hat\gamma\in\N,\bar\gamma\in(0,1];\\
\|f\|_{C^\gamma[s,t]}&=\|f\|_{C^0[s,t]}+[f]_{C^\gamma[s,t]},&\qquad &\gamma>0.
\end{equs}
When the domain is $\R^d$ instead of $[s,t]$, we simply write $C^\gamma$.
When considering functions with values in $L_p(\Omega)$ instead of $\R^d$, $p\in[1,\infty]$, we denote the corresponding spaces and (semi-)norms by $\scC^\gamma_p$.
If $f$ is an adapted process, $\gamma>0$, $0\leq s\leq t\leq 1$, one may choose $Y$ in \eqref{eq:conditional} to be the value at $t$ of the Taylor expansion of $f$ at $s$ up to order $\lfloor\gamma\rfloor$, yielding the bound
\begin{equ}\label{eq:triv bound}
\|f_t-\E^s f_t\|_{L_p(\Omega)}\leq 2 |t-s|^\gamma
[f]_{\scC^\gamma_p[s,t]}.
\end{equ}
Combining with CJI (or with triangle inequality) one also gets, for $0\leq s< u< t\leq 1$,
\begin{equ}\label{eq:triv bound2}
\|\E^u f_t-\E^s f_t\|_{L_p(\Omega)}\leq 2 |t-s|^\gamma
[f]_{\scC^\gamma_p[s,t]}.
\end{equ}
We will also use the following simple property.
Let $f$ be some random process and $\hat f$ a stopped version of it, that is, for some stopping time $\tau$, $\hat f_t=f_{t\wedge\tau}$. Then, for any $\gamma\in(0,1)$, $\eps\in(0,1-\gamma)$, and $p\in(d/\eps, \infty)$, there exists a constant $N$ depending on $\eps$, $p$, and $d$, such that for all $0<t\leq 1$ one has
\begin{equ}\label{eq:stopping}
\,[\hat f]_{\scC^\gamma_p[0,t]}\leq [\hat f]_{L_p(\Omega, C^\gamma[0,t])}\leq
[f]_{L_p(\Omega,C^\gamma[0,t])} \leq N [f]_{\scC^{\gamma+\eps}_p[0,t]}. 
\end{equ}
Indeed, the first two inequalities are trivial, and the third follows from Kolmogorov's continuity theorem.

The ``local intedeterminancy'' property of the processes $B^H$ for $H\in(0,1)$ is a key feature for their regularisation properties, and it also holds in the extended scale.
\begin{proposition}
For any $H\in(0,\infty)\setminus \Z$ there exists a constant $c(H)$ such that
for all $0\leq s\leq t\leq 1$ one has
\begin{equ}\label{eq:BH var}
\E|B_t^H-\E^sB_t^H|^2=dc(H)|t-s|^{2H}.
\end{equ}
\end{proposition}
\begin{proof}
Since the coordinates of $B^H$ are independent, we may and will assume $d=1$.
For $H\in(0,1)$ this is well known, so we assume $\floorH\geq 1$.
From \eqref{eq:BH def} and \eqref{eq:Mandelbrot} we have
\begin{equs}
B_t^H-\E^sB_t^H&=\int_{s\leq r_0\leq r_1\leq\cdots\leq r_{\floorH}\leq t}|r_1-r_0|^{H-\floorH-1/2}\,dW_{r_0}\,dr_1\cdots\,dr_{\floorH}
\\
&=\int_{s\leq r_0\leq r_1\leq\cdots\leq r_{\floorH}\leq t}|r_1-r_0|^{H-\floorH-1/2}\,dr_1\cdots\,dr_{\floorH}\,dW_{r_0}.
\end{equs}
Therefore, by It\^o's isometry,
\begin{equs}
\E|B_t^H-\E^sB_t^H|^2&=\int_s^t\Big(\int_{r_0\leq r_1\leq\cdots\leq r_{\floorH}\leq t}|r_1-r_0|^{H-\floorH-1/2}\,dr_1\cdots\,dr_{\floorH}\Big)^2\,dr_0
\\&=
\int_s^t\Big(\frac{1}{\Pi_{i=1}^{\floorH}(H-i+1/2)}|t-r_0|^{H-1/2}\Big)^2\,dr_0
\\&=
\frac{1}{2H}\Big(\frac{1}{\Pi_{i=1}^{\floorH}(H-i+1/2)}\Big)^2|t-s|^{2H},
\end{equs}
as claimed.
\end{proof}
For any $\eps>0$, $\|B^H\|_{C^{H-\eps}[0,1]}$ is finite almost surely. Indeed, this is well-known for $H\in(0,1)$ and immediately follows for $H>1$ from the definition \eqref{eq:BH def}.
Fix $\eps>0$ such that
\begin{equ}\label{eq:eps}
2(1+H\alpha-H)-\eps\alpha>1,
\end{equ}
which is possible thanks to \eqref{eq:true exponent2}.
For any $K>0$, define the stopping times 
\begin{equ}
\tau_K=\inf\{t:\,\|B^H\|_{C^{H-\eps}[0,t]}>K\}\wedge1.
\end{equ}
From now on we consider the parameter $K$ to be fixed and prove the well-posedness of \eqref{eq:main} up to $\tau_K$.
Correspondingly, we take the following modification of the map $\cT$ from \eqref{eq:main-phi}:
\begin{equ}\label{eq:main-phi-K}
\varphi\mapsto \cT_K(\varphi),\qquad \big(\cT_K(\varphi)\big)_t=\int_0^{t} b(\varphi_{r\wedge\tau_K}+B^H_r)\,dr.
\end{equ}
Denote by $\cS^{0}_K$ the set of Lipschitz continuous and adapted processes and for $k\in\N$ introduce the set of Picard iterates by $\cS^k_K=\cT^{k}_K(\cS^{0}_K)$.

We denote by $\cP_t$ the convolution with Gaussian density whose covariance matrix is $t$ times the identity. In light of \eqref{eq:BH var}, it is natural to use the reparametrisation $\cP^H_t:=\cP_{c(H)t^{2H}}$.
One then has the identity $\E^s f(B^H_t)=\cP^H_{t-s}f(\E^s B^H_t)$.
Let us recall two well-known heat kernel estimates: for $\alpha\in[0,1]$, $\|f\|_{C^\alpha}\leq 1$, $t\in(0,1]$, one has the bounds, with some constant $N$ depending only on $H,\alpha,d$,
\begin{equs}[eq:HK]
|\cP^H_t f(x)-\cP^H_t f(y)|&\leq N t^{H(\alpha-1)}|x-y|\,;
\\
|\cP^H_t f(x_1)-\cP^H_tf(x_2)-\cP^H_tf(x_3)+\cP^H_tf(x_4)|&\leq N\big( t^{H(\alpha-2)}|x_1-x_2||x_1-x_3|
\\&\qquad+t^{H(\alpha-1)}|x_1-x_2-x_3+x_4|\big).
\end{equs}

In the sequel we use $A\lesssim B$ to denote that there exists a constant $N$ such that $A\leq NB$, with $N$ depending only on (a subset of) the following parameters: $H,\alpha, \eps, \|b\|_{C^\alpha}, d, p, K$.

\subsection{Stochastic sewing lemma}
A main tool in the proof is the stochastic sewing lemma introduced in \cite{Khoa}, which however needs to be suitably tweaked to avoid some singular integrals in the application.
Let us first briefly recall the strategy of \cite{Khoa}. Given a filtration $\mathbb{G}=(\G_i)_{i\geq 0}$ and a $\mathbb{G}$- adapted sequence of random variables, $(Z_i)_{i\geq 1}$, one can write the estimate for $p\in [2,\infty)$:
\begin{equs}[eq:SSL trick]
\big\|\sum_{i=1}^n Z_i\big\|_{L_p(\Omega)}
&\leq \big\|\sum_{i=1}^n \E^{\G_{i-1}}Z_i\big\|_{L_p(\Omega)}
+\big\|\sum_{i=1}^n Z_i-\E^{\G_{i-1}}Z_i\big\|_{L_p(\Omega)}
\\
&\lesssim\sum_{i=1}^n \|\E^{\G_{i-1}}Z_i\|_{L_p(\Omega)}
+\Big(\sum_{i=1}^n \|Z_i-\E^{\G_{i-1}}Z_i\|_{L_p(\Omega)}^2\Big)^{1/2},
\end{equs}
where the second sum was estimated via the Burkholder-Davis-Gundy and the Minkowski inequalities. The former can be applied, since the sequence $(Z_i-\E^{\cG_{i-1}}Z_i)_{i\geq 1}$ is one of martingale differences.

Now assume that $(A_{s,t})_{0\leq s<t\leq1}$ is a family of random variables such that $A_{s,t}$ is $\F_t$-measurable. If one is interested in the convergence of the Riemann sums $R^n=\sum_{i=1}^{2^{n}}A_{(i-1)2^{-n},i2^{-n}}$ in $L_p(\Omega)$, then the classical sewing ideas suggest to consider
\begin{equ}
R^{n}-R^{n+1}=\sum_{i=1}^{2^{n}} \delta A_{(i-1) 2^{-n},(i-1/2)2^{-n},i2^{-n}},
\end{equ}
where $\delta A_{s,u,t}=A_{s,t}-A_{s,u}-A_{u,t}$.
This sum fits into the context of \eqref{eq:SSL trick}, with the filtration $\G_i=\F_{i2^{-n}}$. This leads to the requirement in the stochastic sewing lemma on bounding $\E^s\delta A_{s,u,t}$, see \cite[Eq~(2.6)]{Khoa}.

It can be very useful to shift the conditioning backwards in time.
In many applications $A_{s,t}$ is actually $\cF_s$-measurable, which by the above considerations leads to bound $\E^{s-(t-s)/2}\delta A_{s,u,t}$.
Even if this is not the case, one can achieve such a shift by bounding the ``even'' and ``odd'' part of the sum in \eqref{eq:SSL trick} separately, noting that both $(Z_{2i}-\E^{\cG_{2i-2}}Z_{2i})_{i\geq 1}$ and $(Z_{2i+1}-\E^{\cG_{2i-1}}Z_{2i+1})_{i\geq 1}$ are sequences of martingale differences.
To make the formulation cleaner, we assume an extended filtration $(\G_i)_{i\geq-1}$, with which one has the following variant of \eqref{eq:SSL trick}:
\begin{equ}\label{eq:SSL trick2}
\big\|\sum_{i=1}^n Z_i\big\|_{L_p(\Omega)}
\lesssim\sum_{i=1}^n \|\E^{\G_{i-2}}Z_i\|_{L_p(\Omega)}
+\Big(\sum_{i=1}^n \|Z_i-\E^{\G_{i-2}}Z_i\|_{L_p(\Omega)}^2\Big)^{1/2}.
\end{equ}
Since separating the sum of $Z_i$-s into $2$ was of course arbitrary, one can get the same bound with $\G_{i-\ell}$ in place of $\G_{i-2}$ above for any $\ell$, as long as the filtration $\mathbb{G}$ extends to negative indices up to $-\ell+1$.
Another technical modification that we will need is to restrict the triples $(s,u,t)$ on which $\delta A_{s,u,t}$ is considered, to triples where the distances between each pair is comparable. To this end define, for $0\leq S<T\leq1$ and $M\geq0$,
\begin{equs}
\,[S,T]_M^2&=\{(s,t):\,S\leq s<t\leq T,\,s-M(t-s)\geq S\}\,;\\
\,[S,T]^3_M&=\{(s,u,t):\,(s,t)\in[S,T]_M^2,\, u\in(S,T)\,\}\,;\\
\overline{[S,T]}^3_M&=\{(s,u,t)\in [S,T]^3_M:\,(u-s),(t-u)\geq (t-s)/3\}\,.
\end{equs}
The version of the stochastic sewing lemma suitable for our purposes then reads as follows.
\begin{lemma}\label{lem:shifted SSL}
Let $0\leq S<T\leq1$, $p\in[2,\infty)$, $M\geq 0$, and let $(A_{s,t})_{(s,t)\in[S,T]_M^2}$ be a family of random variables in $L_p(\Omega,\R^d)$ such that $A_{s,t}$ is $\F_t$-measurable.
Suppose that for some $\eps_1,\eps_2>0$ and $C_1,C_2$ the bounds
\begin{equs}
\|A_{s,t}\|_{L_p(\Omega)} & \leq C_1|t-s|^{1/2+\eps_1}\,,\label{SSL1}
\\
\|\E^{s-M(t-s)}\delta A_{s,u,t}\|_{L_p(\Omega)} & \leq C_2 |t-s|^{1+\eps_2}\label{SSL2}
\end{equs}
hold for all $(s,t)\in[S,T]_M^2$ and $(s,u,t)\in\overline{[S,T]}_M^3.$
Then there exists a unique (up to modification) adapted 
process $\A:[S,T]\to L_p(\Omega,\R^d)$ such that $\A_S=0$ and such that for some constants $K_1,K_2<\infty$, depending only on $\eps_1,\eps_2$, $p$, $d$, and $M$, the bounds
\begin{align}
\|\A_t	-\A_s-A_{s,t}\|_{L_p(\Omega)} & \leq K_1C_1 |t-s|^{1/2+\eps_1}+K_2C_2 |t-s|^{1+\eps_2}\,,\label{SSL1 cA}
\\
\|\E^{s-M(t-s)}\big(\A_t	-\A_s-A_{s,t}\big)\|_{L_p(\Omega)} & \leq K_2 C_2|t-s|^{1+\eps_2}\label{SSL2 cA}
\end{align}
hold for $(s,t)\in [S,T]_M^2$ and the bound
\begin{equation}\label{SSL3 cA}
\|\A_t-\A_s\|_{L_p(\Omega)}  \leq  K_1C_1 |t-s|^{1/2+\eps_1}+K_2C_2 |t-s|^{1+\eps_2}.
\end{equation}
holds for all $(s,t)\in[S,T]_0^2$.
Moreover, if there exists any continuous process $\tilde\cA:[S,T]\to L_p(\Omega,\R^d)$, $\eps_3>0$, and $K_3<\infty$, such that $\tilde\cA_S=0$ and
\begin{equ}\label{eq:SSL-crit}
\|\tilde\cA_{t}-\tilde\cA_s-A_{s,t}\|_{L_p(\Omega)}\leq K_3|t-s|^{1+\eps_3}
\end{equ}
holds for all $(s,t)\in[S,T]_M^2$, then $\tilde\cA_t=\cA_t$ for all $S\leq t\leq T$.
\end{lemma}
\begin{proof}
The proof essentially follows \cite{Khoa}, keeping in mind the preceding remarks, and making sure that one only ever uses suitably regular partitions.
For the present proof we understand the proportionality constant in $\lesssim$ to depend on $\eps_1,\eps_2,p,d,M$.

Take $(s,t)\in[S,T]_M^2$ and notice that $[s,t]_0^2\subset [S,T]^2_M$.
For a partition $\pi=\{s=t_0<t_1<\cdots<t_n=t\}$ denote $[\pi]=\frac{t-s}{n}$, $|\pi|=\max\{t_i-t_{i-1}\}$, and $\Delta\pi=\min\{t_i-t_{i-1}\}$. Introduce the ``regular'' partitions $\Pi=\Pi_{[s,t]}=\{\pi:\Delta\pi\geq|\pi|/2\}$.
Let us then define
\begin{equ}
A^{\pi}_{s,t}=\sum A_{t_{i-1},t_i}.
\end{equ}
We claim the following.
\begin{enumerate}[(i)]
\item If $\pi\in\Pi$ is a refinement of $\pi'\in\Pi$, then one has the bounds
\begin{align}
\|A^{\pi}_{s,t}-A^{\pi'}_{s,t}\|_{L_p(\Omega)} & \lesssim C_1 |t-s|^{1/2}[\pi']^{\eps_1}+C_2 |t-s|[\pi']^{\eps_2}\,,\label{eq:claim i 1}
\\
\|\E^{s-M(t-s)}\big(A^{\pi}_{s,t}-A^{\pi'}_{s,t}\big)\|_{L_p(\Omega)} & \lesssim C_2|t-s|[\pi']^{\eps_2};\label{eq:claim i 2}
\end{align}
\item The limit $\hat\cA_{s,t}:=\lim_{\pi\in\Pi,|\pi|\to 0}A^{\pi}_{s,t}$ in $L_p(\Omega)$ exists;
\item The family $(\hat\cA_{s,t})_{(s,t)\in[S,T]_M^2}$ is additive, i.e. $\delta\hat \cA_{s,u,t}=0$;
\item The bounds \eqref{SSL1 cA}-\eqref{SSL2 cA} hold with $\hat\cA_{s,t}$ in place of $\cA_t-\cA_s$.
\end{enumerate}

Since for any two $\pi,\pi'\in\Pi$ one can easily find a common refinement $\pi''\in\Pi$, (ii) follows from (i).
Note that for $s\leq u\leq t$ one can find a sequence $(\pi_n)_n\in\Pi_{[s,t]}$ such that $|\pi_n|\to0$, $\pi_n=\pi'_n\cup\pi''_n$, $(\pi'_n)_{n\in\N}\subset\Pi_{[s,u]}$, and $(\pi''_n)_{n\in\N}\subset\Pi_{[u,t]}$. Therefore (iii) is immediate from (ii). Finally, (iv) follows from (i) by setting $\pi'=\{s,t\}$ and letting $|\pi|\to 0$. 

Therefore, it only remains to show (i). We define a map $\rho$ on $\Pi$ as follows. If $ \pi=\{t_0,t_1,\ldots ,t_{n}\}$
and $n$ is odd, first set $n'=n+1$ and $t'_i=(t_i+t_{i-1})/2$, where $i$ is the first index such that $|\pi|=t_i-t_{i-1}$. Then set $t'_j=t_j$ for $0\leq j<i$ and $t'_j=t_{j-1}$ for $i<j\leq n'$.
If $n$ is even, simply set $n'=n$ and $t'_j=t_j$ for all $0\leq j\leq n$.
Then set $\rho(\pi)=\{t'_0,t'_2,t'_4,\ldots t'_{n'/2}\}$.
It is clear that $\rho(\pi)\in\Pi$ and that $[\rho(\pi)]\geq(3/2)[\pi]$ (unless $\pi$ is the trivial partition). One has
\begin{equ}\label{eq:proof mess}
A^{\pi}_{s,t}-A^{\rho(\pi)}_{s,t}=\delta A_{t_{i-1},t'_i,t_i}- \sum_{j=0}^{(n'-2)/2}\delta A_{t'_{2j},t'_{2j+1},t'_{2j+2}}.
\end{equ}
Note that each triple appearing belongs to $\overline{[S,T]}_M^3$.
The first term is simply bounded via \eqref{SSL1}. For the sum we use \eqref{eq:SSL trick2} and the succeeding remark. Note that $\delta A_{t'_{2j},t'_{2j+1},t'_{2j+2}}$ is $\F_{t'_{2j+2}}$-measurable, and that for large enough $\ell=\ell(M)$ (uniform in $\pi\in\Pi$, $n$, and $j$) one has $t'_{2(j-\ell)+2}\leq t'_{2j}-M(t'_{2j+2}-t'_{2j})$. Indeed, this is a consequence of the regularity of the partitions in $\Pi$. Therefore, the sum
\begin{equ}
\sum_{j=0}^{(n'-2)/2}\delta A_{t'_{2j},t'_{2j+1},t'_{2j+2}}-\E^{t'_{2j}-M(t'_{2j+2}-t'_{2j})}\delta A_{t'_{2j},t'_{2j+1},t'_{2j+2}}
\end{equ}
can be decomposed as the sum of $\ell$ different sums of martingale differences, each of which can be bounded as in \eqref{eq:SSL trick}.
From the bounds \eqref{SSL1 cA}-\eqref{SSL2 cA} one can conclude that \eqref{eq:claim i 1}-\eqref{eq:claim i 2} holds for $\pi'=\rho(\pi)$. Now suppose that $\pi'$ is the trivial partition $\{s,t\}$ and define $m$ as the smallest integer such that $\rho^m(\pi)=\pi'$. One can write
\begin{equs}
\|A^{\pi}_{s,t}-A^{\pi'}_{s,t}\|_{L_p(\Omega)}&\leq\sum_{i=0}^{m-1} \|A^{\rho^i(\pi)}_{s,t}-A^{\rho^{i+1}(\pi)}_{s,t}\|_{L_p(\Omega)}
\\
&\lesssim\sum_{i=0}^{m-1}C_1|t-s|^{1/2}[\rho^{i+1}(\pi)]^{\eps_1}
+C_2|t-s|[\rho^{i+1}(\pi)]^{\eps_2}.
\end{equs}
Recalling $[\rho^i(\pi)]\leq(2/3)^{m-i}[\rho^m(\pi)]=(2/3)^{m-i}[\pi']$,
this yields \eqref{eq:claim i 1} (and one gets similarly \eqref{eq:claim i 2}) for $\pi'=\{s,t\}$. Finally, if $\pi'=\{t_0,\ldots,t_n\}\in\Pi$ is arbitrary and $\pi\in\Pi$ is a refinement of it, define $\pi_i$ to be the restriction of $\pi$ to $[t_i,t_{i+1}]$. It is clear that $\pi_i\in\Pi_{[t_i,t_{i+1}]}$. Therefore writing
\begin{equ}
A_{s,t}^{\pi}-A_{s,t}^{\pi'}=\sum_{i=0}^{n-1}A^{\pi_i}_{t_i,t_{i+1}}-A_{t_i,t_{i+1}},
\end{equ}
each term in the sum is of the form that fits in the previous case, and so admits bounds of the form \eqref{eq:claim i 1}-\eqref{eq:claim i 2}. Using these bounds, the sum is treated just like the one in \eqref{eq:proof mess}, using \eqref{eq:SSL trick2}. Hence, (i) is proved.

From (iv) it is then immediate that \eqref{SSL3 cA} is also satisfied with $\hat\cA_{s,t}$ in place of $\cA_t-\cA_s$, but only for $(s,t)\in[S,T]_M$.
Let us define $t_i=S+(S-T)\big(\tfrac{M}{M+1}\big)^i$ and set, for $t\in[S,T]$,
\begin{equ}
\cA_t=\sum_{i=1}^\infty\hat\cA_{t_i\wedge t,\,t_{i-1}\wedge t}.
\end{equ}
First notice that for all $i\geq 1$ and $t\in[t_i,t_{i-1}]$, one has $(t_i,t)\in[S,T]_M$, so each term in the sum is well-defined.
The convergence of the sum in $L_p(\Omega)$ immediately follows from the bounds on $\hat\cA_{s,t}$ and the geometric decay of $|t_i-t_{i-1}|$.
It is also clear that $\cA$ is adapted and that $\cA_t\to 0$ in $L_p(\Omega)$ as $t\to S$.
Since $\cA_t-\cA_s=\hat \cA_{s,t}$ for pairs $(s,t)\in[S,T]_M$, the bounds \eqref{SSL1 cA}-\eqref{SSL2 cA}-\eqref{SSL3 cA} hold for such pairs. Extending
\eqref{SSL3 cA} to all $(s,t)\in[S,T]_0$ is standard, see e.g. \cite[Lem~2.3]{Carlo} for a very similar statement. The characterisation \eqref{eq:SSL-crit} is also standard, see e.g. \cite[Lem~4.2]{FH}.
\end{proof}

\begin{remark}
The usefulness of these small shifts can be illustrated with the following analogy.
To have $\int_0^{2^{-n}}x^\gamma\,dx\approx 2^{-(\gamma+1)n}$, one requires $\gamma>-1$. However, upon shifting the integral, $\int_{2^{-n}}^{2^{-n+1}}x^\gamma\,dx\approx 2^{-(\gamma+1)n}$ is true for \emph{all} $\gamma\in\R$. This example is actually more than an analogy: in \eqref{eq:I1} below we encounter such integrals and the above formulation allows one to bypass the condition $H(\alpha-2)>-1$.
\end{remark}

\subsection{Higher order expansion of solutions}
As the final ingredient, let us discuss how well a solution $\varphi$ (and more generally, an element of $\cS^k_K$) can be approximated at time $t$ by an $\F_s$-measurable random variable, for $s<t$.
From the differentiability of $\varphi$, one immediately sees that taking $\varphi_s$ gives an approximation of order $|t-s|$, or, slightly more cleverly, taking $\varphi_s+(t-s)\varphi'_s$ gives an approximation of order $|t-s|^{1+\alpha}$.
Since $\varphi$ is certainly \emph{not} twice differentiable, it might be surprising that one can push this expansion \emph{much} further and obtain an approximation of order almost $|t-s|^{1+\alpha H}$.

\begin{lemma}\label{lem:regularity}
Assume the setting and notations of Theorem \ref{thm:main} and Section \ref{sec:notations}, and in addition assume $\alpha<1$.
Then there exists a $k_0=k_0(H,\alpha,\eps)\in\N$ such that for all $\varphi\in \cS^{k_0}_K$ and
for all $0\leq s\leq 1$ one has almost surely for all $t\in[s,\tau_K]$
\begin{equs}
|\varphi_t-\E^s\varphi_t|&\lesssim |t-s|^{1+\alpha (H-\eps)}.\label{eq:cool bound}
\end{equs}
\end{lemma}
\begin{proof}
Let
\begin{equ}
\bB_{s,t}^H=\sum_{i=0}^{\floorH}\frac{(t-s)^i}{i!}\d^i B^H_s.
\end{equ}
Clearly, for $0\leq s\leq t\leq\tau_K$, one has $|B^H_t-\bB^H_{s,t}|\lesssim |t-s|^{H-\eps}$.
Define maps $\bA^{(k)}$ on $\cS^k_K$ as follows. For $k=0$ and $\varphi\in \cS^0_K$ set $\bA^{(0)}_{s,t}\varphi=\varphi_s$. For $k>0$ and $\varphi\in \cS^k_K$, take $\psi\in \cS^{k-1}_K$ such that $\varphi=\cT_K(\psi)$ and define inductively
\begin{equ}
\bA_{s,t}^{(k)}\varphi=\varphi_s+\int_s^t b\big(\bA_{s,r}^{(k-1)}\psi+\bB^H_{s,r}\big)\,dr.
\end{equ}
We aim to get almost sure bounds of the form
\begin{equ}\label{eq:Phi approx}
|\varphi_t-\bA_{s,t}^{(k)}\varphi|\lesssim |t-s|^{\gamma_k},\qquad 0\leq s\leq t\leq\tau_K
\end{equ}
with some $\gamma_k$ to be determined.
We proceed by induction.
Clearly \eqref{eq:Phi approx} holds for $k=0$ with $\gamma_0=1$. In the inductive step, from \eqref{eq:Phi approx} one deduces 
\begin{equs}
|\varphi_t-\bA_{s,t}^{(k)}\varphi|&=\big|\int_s^t b(\psi_r+B^H_r)-b\big(\bA_{s,r}^{(k-1)}\psi+\bB^H_{s,r}\big)\,dr\big|
\\
&\lesssim \int_s^t |\psi_r-\bA_{s,r}^{(k-1)}\psi|^\alpha+|B^H_r-\bB^H_{s,r}|^\alpha\,dr
\\
&\lesssim |t-s|^{1+\alpha\big(\gamma_{k-1}\wedge(H-\eps)\big)},
\end{equs}
where in the last inequality we used the induction hypothesis.
Therefore, \eqref{eq:Phi approx} holds for all $k\in\N$, with $\gamma_k$ defined by the recursion
\begin{equ}
\gamma_0=1,\qquad
\gamma_{k}=1+\alpha\big(\gamma_{k-1}\wedge(H-\eps)\big).
\end{equ}
For $k< k_0:=\inf\{\ell\in\N:\,\gamma_\ell>H-\eps\}+1$, one has $\gamma_{k-1}\leq H-\eps$ and therefore
\begin{equ}
\gamma_{k}=1+\alpha+\alpha^2+\cdots+\alpha^{k}=\frac{1-\alpha^{k+1}}{1-\alpha}\underset{k\to\infty}{\longrightarrow}\frac{1}{1-\alpha}>H.
\end{equ}
Hence, $k_0$ is finite. One then has
\begin{equ}
\gamma_{k_0}=1+\alpha\big(\gamma_{k_0-1}\wedge(H-\eps)\big)
=1+\alpha(H-\eps).
\end{equ}
Rewriting \eqref{eq:Phi approx} in an equivalent form,
one has almost surely for all $0\leq s\leq t\leq 1$
\begin{equ}
|\varphi_{t\wedge\tau_K}-\bA_{s\wedge\tau_K,t\wedge\tau_K}^{(k_0)}\varphi|\lesssim |t-s|^{1+\alpha(H-\eps)}.
\end{equ}
For fixed $s$ and $t$, apply \eqref{eq:conditional} with $p=\infty$, $X=\varphi_{t\wedge\tau_K}$, and $Y=\bA_{s\wedge\tau_K,t\wedge\tau_K}^{(k_0)}\varphi$, since the latter is certainly measurable with respect to $\F_s$.
Therefore,
one has almost surely
\begin{equ}
|\varphi_{t\wedge\tau_K}-\E^s\varphi_{t\wedge\tau_K}|\lesssim |t-s|^{1+\alpha(H-\eps)}.
\end{equ}
From the continuity of $\varphi$, this bound holds
for all $0\leq s\leq 1$ almost surely for all $t\in[s,1]$,
which readily implies the claim.
\end{proof}

\begin{remark}
One can easily check that as far as \eqref{eq:cool bound} is concerned, condition \eqref{eq:true exponent} is overkill: indeed, the above argument works as long as $H>1$ and
\begin{equ}\label{eq:weak exponent}
\alpha>1-\frac{1}{H}\,.
\end{equ} 
Condition \eqref{eq:weak exponent} is also equivalent to the existence of $\eps>0$ such that $1+\alpha(H-\eps)>H$. In other words, \eqref{eq:weak exponent} is precisely the condition that guarantees that the drift components of sufficiently high order Picard iterates (and in particular, solutions, if any exist) are ``more regular'', in a stochastic sense, than their noise components. 
While this is only a heuristic indication of regularisation effects, it is interesting to note that several works in the literature connect the condition \eqref{eq:weak exponent} to \emph{weak} well-posedness.

As for necessity, it is shown in \cite{dR-example} that weak uniqueness fails for $\alpha<1-1/H$.
As for sufficiency, a number of results are available in Markovian settings.
In the standard Brownian (that is, $H=1/2$) case \cite{DD, CC} constructed martingale solutions to \eqref{eq:main} with (so-called ``enhanced'') distributional drift of regularity $\alpha>-2/3$,
but the power counting heuristics work all the way to $\alpha>-1=1-1/H$.
Similarly, \cite{HP20} studied martingale solutions for SDEs driven by $\lambda$-stable L\'evy noise in the range $\alpha>(2-2\lambda)/3$, with the power counting suggesting the threshold $1-\lambda$. This is also consistent with \eqref{eq:weak exponent}, using the scaling correspondence $\lambda\leftrightarrow 1/H$.
In the degenerate Brownian (that is, $H=k+1/2$, $k\in\N$) case weak well-posedness under the condition \eqref{eq:weak exponent} was proved in \cite{dR-weak}.

It is therefore natural to conjecture that \eqref{eq:weak exponent} guarantees weak well-posedness also in the non-Markovian case, but to our best knowledge this question is still well open.
\end{remark}

\section{Proof of Theorem \ref{thm:main}}\label{sec:proof}
Take $p\in[2,\infty)$, to be specified to be large enough later.  Recall that it suffices to prove the well-posedness of \eqref{eq:main} up to $\tau_K$.
To this end, we will show that $\cT_K$ is a contraction on $\cS^{k_0}_K$ in a suitable metric. For any process $f$, we define the stopped process $f^K$ by $f^K_t=f_{t\wedge\tau_K}$.

We want to apply Lemma \ref{lem:shifted SSL}, with $M=1$.
Let $\varphi,\psi\in \cS^{k_0}_K$.
Fix $0\leq S<T\leq 1$ and for $(s,t)\in[S,T]_{1}^2$
define
\begin{equ}
A_{s,t}=\E^{s-(t-s)}\int_{s}^{t} b(B^H_r+\E^{s-(t-s)}\varphi_{r}^K)-b(B^H_r+\E^{s-(t-s)}\psi_{r}^K)\,dr.
\end{equ}
For $(s,t)\in[S,T]_1^2$ denote $s_1=s-(t-s),s_2=s,s_3=t$. Then by \eqref{eq:HK} and CJI one has
\begin{equs}
\|A_{s,t}\|_{L_p(\Omega)}&\lesssim\big\|\int_{s_2}^{s_3}\cP_{r-s_1}^H
b(\E^{s_1} B^H_r+\E^{s_1}\varphi_{r}^K) -\cP_{r-s_1}^H
b(\E^{s_1} B^H_r+\E^{s_1}\psi_{r}^K)\,dr\big\|_{L_p(\Omega)}
\\
&\lesssim\int_{s_2}^{s_3} (r-s_1)^{H(\alpha-1)}
\|\E^{s_1}\varphi_{r}^K-\E^{s_1}\psi_{r}^K\|_{L_p(\Omega)}\,dr
\\
&\lesssim |t-s|^{1+H\alpha-H}\|\varphi^K-\psi^K\|_{\scC^0_p[s_2,s_3]}.
\end{equs}
From \eqref{eq:true exponent2}, the condition \eqref{SSL1} is satisfied with $C_1=\|\varphi^K-\psi^K\|_{\scC^0_p[0,T]}$.

Concerning the second condition of Lemma \ref{lem:shifted SSL}, for $(s,u,t)\in\overline{[S,T]}^3_1$ denote $s_1=s-(t-s)$, $s_2=s-(u-s)$, $s_3=u-(t-u)$, $s_4=s$, $s_5=u$, $s_6=t$.
These points are \emph{not} necessarily ordered according to their indices, but thanks to the definition of $\overline{[S,T]}^3_1$ they satisfy $(s_4-s_2),(s_5-s_3)\geq (t-s)/3$.
One can write
\begin{equs}
\E^{s-(t-s)}\delta A_{s,u,t}
&=
\E^{s_1}\E^{s_2}\int_{s_4}^{s_5}b(B^H_r+\E^{s_1}\varphi_{r}^K)
-b(B^H_r+\E^{s_1}\psi_{r}^K)
\\&\qquad\qquad\qquad-b(B^H_r+\E^{s_2}\varphi_{r}^K)+b(B^H_r+\E^{s_2}\psi_{r}^K)\,dr
\\
&\quad+\E^{s_1}\E^{s_3}\int_{s_5}^{s_6}b(B^H_r+\E^{s_1}\varphi_{r}^K)-b(B^H_r+\E^{s_1}\psi_{r}^K)
\\
&\qquad\qquad\qquad-b(B^H_r+\E^{s_3}\varphi_{r}^K)+b(B^H_r+\E^{s_3}\psi_{r}^K)\,dr
\\
&=:I_1+I_2.
\end{equs}
The two terms are treated in exactly the same way, so we only detail $I_1$.
By \eqref{eq:HK},
\begin{equs}[eq:I1]
\|I_1\|_{L_p(\Omega)}&\lesssim
\Big\|\int_{s_4}^{s_5}|r-s_2|^{H(\alpha-2)}|\E^{s_1}\varphi_r^K-\E^{s_1}\psi_r^K||\E^{s_1}\varphi_r^K-\E^{s_2}\varphi_r^K|
\\&\qquad\qquad+|r-s_2|^{H(\alpha-1)}|\E^{s_1}\varphi_r-\E^{s_1}\psi_r-\E^{s_2}\varphi_r+\E^{s_2}\psi_r|\,dr\Big\|_{L_p(\Omega)}.
\end{equs}
From CJI, \eqref{eq:cool bound}, and \eqref{eq:triv bound2}, we have the bounds
\begin{equs}
&\|\E^{s_1}\varphi_r^K-\E^{s_1}\psi_r^K\|_{L_p(\Omega)}\leq\|\varphi^K-\psi^K\|_{\scC^0_p[s_4,s_5]}\,;
\\
&|\E^{s_1}\varphi_r^K-\E^{s_2}\varphi_r^K|\leq |\E^{s_1}\varphi_r^K-\varphi_r^K|+|\E^{s_2}\varphi_r^K-\varphi_r^K|\lesssim |t-s|^{1+\alpha(H-\eps)}\,;
\\
&\|\E^{s_1}(\varphi^K-\psi^K)_r-\E^{s_2}(\varphi^K-\psi^K)_r\|_{L_p(\Omega)}
\lesssim|t-s|^{1/2}[\varphi^K-\psi^K]_{\scC^{1/2}_p[s_1,s_5]}\,.
\end{equs}
Substituting this into \eqref{eq:I1}, and recalling that $\int_{s_4}^{s_5}|r-s_2|^\gamma\,dr\lesssim |t-s|^{1+\gamma}$ for any $\gamma\in\R$, we get
\begin{equs}
\,&\|I_1\|_{L_p(\Omega)}
\\&\lesssim|t-s|^{1+H(\alpha-2)+1+\alpha (H-\eps)}\|\varphi^K-\psi^K\|_{\scC^0_p[s_4,s_5]}+|t-s|^{1+H(\alpha-1)+1/2}[\varphi^K-\psi^K]_{\scC^{1/2}_p[s_1,s_5]}
\\
&=|t-s|^{2(1+H\alpha-H)-\eps \alpha}\|\varphi^K-\psi^K\|_{\scC^0_p[s_4,s_5]}+|t-s|^{(1+H\alpha-H)+1/2}[\varphi^K-\psi^K]_{\scC^{1/2}_p[s_1,s_5]}.
\end{equs}
From \eqref{eq:true exponent2} and \eqref{eq:eps}, we see that both exponents of $|t-s|$ above are strictly bigger than $1$.
With the analogous bound on $I_2$, one concludes that \eqref{SSL2} holds with $C_2=\|\varphi^K-\psi^K\|_{\scC^0_p[S,T]}+[\varphi^K-\psi^K]_{\scC^{1/2}_p[S,T]}$,
and therefore Lemma \ref{lem:shifted SSL} applies. 
We claim that the process $\cA$ is given by
\begin{equ}
\tilde\cA_t=\big(\cT_K(\varphi)-\cT_K(\psi)\big)_t=\int_0^{t} b(B^H_r+\varphi_{r}^K)-b(B^H_r+\psi_{r}^K)\,dr.
\end{equ}
Indeed, first notice that that since both $B^H$, $\varphi^K$, and $\psi^K$ are all trivially Lipschitz-continuous with their Lipschitz constant having $p$-th moments, $\tilde \cA$ belongs to $\scC^{1+\alpha}_p$.
Therefore, by \eqref{eq:triv bound} we can write 
\begin{equ}
\|\tilde\cA_{t}-\tilde\cA_s-\E^{s-(t-s)}(\tilde\cA_{t}-\tilde\cA_s)\|_{L_p(\Omega)}\lesssim |t-s|^{1+\alpha}\|\tilde\cA\|_{\scC^{1+\alpha}_p}\lesssim|t-s|^{1+\alpha}.
\end{equ}
On the other hand, by CJI and \eqref{eq:triv bound} again, we have
\begin{equs}
\|\E^{s-(t-s)}(\tilde\cA_{t}-\tilde\cA_s)-A_{s,t}\|_{L_p(\Omega)}&\lesssim 
\int_{s}^{t}\|b(B^H_r+\E^{s-(t-s)}\varphi_r^K)-b(B^H_r+\varphi_r^K)\|_{L_p(\Omega)}
\\&\qquad+
\|b(B^H_r+\E^{s-(t-s)}\psi_r^K)-b(B^H_r+\psi_r^K)\|_{L_p(\Omega)}\,dr
\\
&\lesssim |t-s|^{1+\alpha}.
\end{equs}
Hence \eqref{eq:SSL-crit} is satisfied, and $\tilde\cA=\cA$.
The bound \eqref{SSL3 cA} then yields, with some $\eps'>0$, for all $S\leq s<t\leq T$,
\begin{equs}
\|\cA_{t}-\cA_s\|_{L_p(\Omega)}&=\|\big(\cT_K(\varphi)-\cT_K(\psi)\big)_t-\big(\cT_k(\varphi)-\cT_k(\psi)\big)_s\|_{L_p(\Omega)}
\\
&\lesssim|t-s|^{1/2+\eps'}\|\varphi^K-\psi^K\|_{\scC^0_p[S,T]}+
|t-s|^{1+\eps'}\|\varphi^K-\psi^K\|_{\scC^0_p[S,T]}
\\
&\qquad+|t-s|^{1+\eps'}[\varphi^K-\psi^K]_{\scC^{1/2}_p[S,T]}
\\
&\lesssim|t-s|^{1/2+\eps'}\|\varphi^K-\psi^K\|_{\scC^0_p[S,T]}
+|t-s|^{1+\eps'}[\varphi^K-\psi^K]_{\scC^{1/2}_p[S,T]}.\qquad\label{eq:000}
\end{equs}
Let us first choose $S=0$, in which case $\|\varphi^K-\psi^K\|_{\scC^0_p[0,T]}\lesssim[\varphi^K-\psi^K]_{\scC^{1/2}_p[0,T]}$. Therefore, upon dividing by $|t-s|^{1/2+\eps'/2}$ and taking supremum over $0\leq s<t\leq T$, one gets
\begin{equ}
\,[\cT_K(\varphi)-\cT_K(\psi)]_{\scC^{1/2+\eps'/2}_p[0,T]}\lesssim T^{\eps'/2}[\varphi^K-\psi^K]_{\scC^{1/2}_p[0,T]}.
\end{equ}
We can apply \eqref{eq:stopping} with $f=\cT_K(\varphi)-\cT_K(\psi)$, $\tau=\tau_K$, $\gamma=1/2$, $\eps=\eps'/2$. This is the only point where the choice of $p$ matters: we need to take $p>2d/\eps'$.
We get
\begin{equ}
\,\big[\big(\cT_K(\varphi)\big)^K-\big(\cT_K(\psi)\big)^K\big]_{\scC^{1/2}_p[0,T]}\lesssim T^{\eps'/2}[\varphi^K-\psi^K]_{\scC^{1/2}_p[0,T]},
\end{equ}
and therefore, for sufficiently small $T>0$,
\begin{equ}\label{eq:contraction}
\,\big[\big(\cT_K(\varphi)\big)^K-\big(\cT_K(\psi)\big)^K\big]_{\scC^{1/2}_p[0,T]}\leq \frac{1}{2}[\varphi^K-\psi^K]_{\scC^{1/2}_p[0,T]}.
\end{equ}
We conclude that $\cT_K$ is a contraction on $\cS^{k_0}_K$ with respect to $[(\cdot)^K]_{\scC^{1/2}_p[0,T]}$.

From here the argument is more or less routine, with some tedium due to the stopping times.
Introduce $\hat\cT_K$ by setting $\hat\cT_K(\lambda)=\big(\cT_K(\lambda)\big)^K$ and let $\hat\cS_K^n=\hat\cT_K^n(\cS_K^0)$.
It is clear from the definition \eqref{eq:main-phi-K} that for any $n\geq 1$, $\hat\lambda\in\hat\cS_K^n$ if and only if $\hat\lambda=\lambda^K$ for some $\lambda\in\cS_K^n$.
In particular, \eqref{eq:contraction} can be written as follows: for any $\hat\varphi,\hat\psi\in\hat\cS_K^{k_0}$,
\begin{equ}\label{eq:contraction 2}
\,\big[\hat\cT_K(\hat\varphi)-\hat\cT_K(\hat\psi)\big]_{\scC^{1/2}_p[0,T]}\leq \frac{1}{2}[\hat\varphi-\hat\psi]_{\scC^{1/2}_p[0,T]}.
\end{equ}
Notice that on the subset $\hat\scC$ of $\scC^{1/2}_p[0,T]$ consisting of processes that vanish at time $0$, the  seminorm $[\cdot]_{\scC^{1/2}_p[0,T]}$ is actually a norm.
Therefore for any fixed $\hat\varphi\in \hat\cS^{k_0}_K$, the sequence $(\hat\cT^n_K(\hat\varphi))_{n\in\N}$ converges in $\hat\scC$ to some $\hat\varphi^\ast$. Since $\hat\cT_K$ is continuous on $\hat\scC$, one has $\hat\varphi^\ast=\hat\cT_K(\hat\varphi^\ast)$ on $[0,T]$, yielding a local solution.

To obtain a solution beyond time $T$, let $\cS^{0,\ast}_K$ to be the set of adapted Lipschitz continuous processes agreeing with $\hat\varphi^\ast$ on $[0,T]$
and set
$\cS^{k,\ast}_K=\cT_K^{k}(\cS^{0}_K)$,
$\hat\cS^{k,\ast}_K=\hat\cT_K^{k}(\cS^{0}_K)$.
Note that elements of $\cS^{k,\ast}_K$ agree with $\hat\varphi^\ast$ on $[0,\tau_K\wedge T]$, while elements of $\hat\cS^{k,\ast}_K$ agree with $\hat\varphi^\ast$ on $[0,T]$.
Also, trivially, $\cS^{k,\ast}_K\subset\cS^{k}_K$
and therefore the bound \eqref{eq:000} holds for $\varphi,\psi\in\cS^{k_0,\ast}_K$.
Moreover, notice that on $\cS^{k,\ast}_K$, one has $\|\varphi^K-\psi^K\|_{\scC^0_p[T,2T]}\lesssim[\varphi^K-\psi^K]_{\scC^{1/2}_p[T,2T]}$. 
Therefore similarly to \eqref{eq:contraction}, we get
\begin{equ}
\,\big[\big(\cT_K(\varphi)\big)^K-\big(\cT_K(\psi)\big)^K\big]_{\scC^{1/2}_p[T,2T]}\leq \frac{1}{2}[\varphi^K-\psi^K]_{\scC^{1/2}_p[T,2T]}.
\end{equ}
Following the same argument as above, we obtain a process $\hat\varphi^{\ast\ast}$ such that $\hat\varphi^{\ast\ast}=\hat\cT_K(\hat\varphi^{\ast\ast})$ on $[T,2T]$ and that furthermore agrees with $\hat\varphi^\ast$ on $[0,T]$. Therefore, in fact, one has $\varphi^{\ast\ast}=\hat\cT_K(\hat\varphi^{\ast\ast})$ on $[0,2T]$. Repeating this argument finitely many times yields a fixed point of $\hat\varphi^{\star}$ on $[0,1]$.

The uniqueness is immediate from the above construction: any two fixed points of $\hat\cT_K$ necessarily lie in $\hat\cS^{k_0}_K$ and therefore they agree on $[0,T]$ by \eqref{eq:contraction 2}. One can then proceed iteratively as above.
Therefore, $X=\hat\varphi^{\star}+B^H$ is the unique solution of \eqref{eq:main} up to $\tau_K$, finishing the proof. 
\qed

\bigskip
\textbf{Acknowledgment.} The author thanks Oleg Butkovsky for discussions on the topic and for pointing out the references \cite{Brownian-1, Brownian-2}, and Lucio Galeati for pointing out a mistake in an earlier version of the paper.

\bibliographystyle{Martin}
\bibliography{smooth}

\end{document}